\documentclass[11pt]{amsart}

\usepackage[utf8]{inputenc}
\usepackage[english]{babel}
\usepackage{amsmath}
\usepackage{amssymb}
\usepackage[pdftex]{graphicx}
\usepackage{amsthm}
\usepackage{enumerate}
\usepackage[numbers,sort&compress]{natbib}
\usepackage{amsthm}
\usepackage[pdfborder={0 0 0}]{hyperref}
\usepackage{bm}
\usepackage{turnstile}
\usepackage{color}
\usepackage{setspace} 
\usepackage{comment}
\usepackage{centernot,etoolbox}

\usepackage[textsize=footnotesize,color=yellow!40, bordercolor=white]{todonotes}

\newcommand{\comm}[1]{}

\DeclareRobustCommand{\SkipTocEntry}[5]{}

\newtheorem{theorem}{Theorem}[section]
\newtheorem*{theorem*}{Theorem} 
\newtheorem{lemma}[theorem]{Lemma}

\newtheorem{proposition}[theorem]{Proposition}
\newtheorem{conjecture}[theorem]{Conjecture}

\newtheorem{question}[theorem]{Question}

\theoremstyle{definition}

\newtheorem{definition}[theorem]{Definition}

\theoremstyle{remark}
\newtheorem{remark}[theorem]{Remark}
\theoremstyle{definition}
\newcounter{cl}
\AtBeginEnvironment{theorem}{\setcounter{cl}{0}}
\AtBeginEnvironment{lemma}{\setcounter{cl}{0}}
\AtBeginEnvironment{remark}{\setcounter{cl}{0}}
\newtheorem{claim}[cl]{Claim}%[theorem]
\newtheorem*{claim*}{Claim}
\newtheorem{subclaim}{Subclaim}%[theorem]
\newtheorem*{subclaim*}{Subclaim}
\theoremstyle{remark}

\newtheorem*{case*}{Case}

\newtheorem*{subcase*}{Subcase}

\newcommand{\crit}{\ensuremath{\operatorname{crit}} }

\newcommand{\lh}{\ensuremath{\operatorname{lh}} }

\newcommand{\Ult}{\ensuremath{\operatorname{Ult}}}
\newcommand{\tUlt}{\ensuremath{\overline{\operatorname{Ult}}}}
\newcommand{\nlhd}{\centernot{\lhd}}

\newcommand{\st}{\mid}

\newcommand{\cf}{\ensuremath{\operatorname{cf}}}
\newcommand{\id}{\ensuremath{\operatorname{id}}}

\newcommand{\dom}{\ensuremath{\operatorname{dom}}}

\newcommand{\mo}{\triangleleft}
\renewcommand{\l}{\langle}
\renewcommand{\r}{\rangle}
\newcommand{\uhr}{\upharpoonright}
\newcommand{\res}{\uhr}

\title{Infinite decreasing chains in the Mitchell order}

\author{Omer Ben-Neria}
\address{Omer Ben-Neria, Einstein Institute of Mathematics, The Hebrew University of Jerusalem. Jerusalem 91904, Israel.}
\email{omer.bn@mail.huji.ac.il}

\author{Sandra M\"uller}
\address{Sandra M\"uller, Institut f\"ur Mathematik, Universit\"at
  Wien. Kolingasse 14-16, 1090 Wien, Austria.}
\email{mueller.sandra@univie.ac.at}

\date{\today}

\begin{document}

\maketitle

\begin{abstract}
  It is known that the behavior of the Mitchell order substantially changes at the level of rank-to-rank extenders, as it ceases to be well-founded. While the possible partial order structure of the Mitchell order below rank-to-rank extenders is considered to be well understood, little is known about the structure in the ill-founded case. The purpose of the paper is to make a first step in understanding this case, by studying the extent to which the Mitchell order can be ill-founded.
  Our main results are (i) in the presence of a rank-to-rank extender there is a transitive Mitchell order decreasing sequence of extenders of any countable length, and (ii) there is no such sequence of length $\omega_1$.
\end{abstract}

\section{Introduction}

Extenders are combinatorial objects which play a fundamental role in capturing the strength of various large cardinal axioms, and specifically in capturing the strength of elementary embeddings. 
Given an elementary embedding $j : V \rightarrow M$ with critical point $\kappa$, and an ordinal $\lambda > \kappa$, the \emph{$(\kappa,\lambda)$-extender $E$ derived from $j$} is a system of ultrafilters  
$E = (E_a \st a \in [\lambda]^{<\omega})$, where each $E_a$ is
given by 
\[ X \in E_a \text{ iff } X \subseteq [\xi]^{|a|} \wedge a \in j(X), \] 
where  $\xi \geq \kappa$ is the least ordinal such that $\lambda \leq j(\xi)$. 
For $a \in [\lambda]^{<\omega}$, $\mu_a$
  denotes the least $\mu$ such that $a \subseteq j(\mu)$.
  Suppose that $a \subseteq b$ are elements of $[\lambda]^{<\omega}$, 
  where $|b| = n$ and $I \subseteq n$ is the set of indices (given by the canonical order on $\lambda$) for which $b \uhr I = a$. 
  We note that the map $\pi_{b,a} : [\mu_b]^{|b|} \to [\mu_a]^{|a|}$ given by 
  $\pi_{b,a}(x) = x\uhr I$ satifies that for every $X \in E_a$,  $\pi_{b,a}^{-1}(X) \in E_b$. 
Each ultrafilter $E_a$ is $\kappa$-complete, and the system of ultrafilters naturally gives rise to a canonical system of ultrapowers $\Ult(V,E_a)$.
The manner we derived each $E_a$ from $j$ implies that $\Ult(V,E_a)$ is isomorphic to the subclass $X_a$ of $M$ given by $X_a = \{ j(f)(a) \mid f : [\mu_a]^{|a|} \to V\}$.
The functions $\pi_{b,a}$ allow us to form a natural direct limit of the system of ultrapowers, denoted by $\Ult(V,E)$, which is isomorphic to the subclass $X_E = \bigcup_{a \in [\lambda]^{<\omega}} X_a$ of $M$. 
We denote the transitive collapse of $\Ult(V,E)$ (which is the same as the transitive collapse of $X_E$) and the resulting ultrapower embedding by $i_E : V \to M_E$. It follows from the description of $X_E$ that if the original  embedding $j : V \to M$  is $\lambda$-strong (i.e., $V_\lambda \subseteq M$) then $V_\lambda \subseteq M_E$. In that sense, we see that the ultrapower of $V$ by the extender $E$ captures the strength of $j : V \to M$ up to $\lambda$. The properties of an extender $E = (E_a, \pi_{b,a} \st a \subseteq b \in [\lambda]^{<\omega})$ as a system of ultrafilters with suitable connecting maps can be formulated directly, without the need of an ambient embedding $j : V \to M$. 
We refer the reader to \cite[Section 26]{Ka08} for an extensive introduction to the theory of extenders.\\

\noindent
A fundamental notion in the study of extenders is the one of the Mitchell order. 

\begin{definition} Let $E,E'$ be two extenders. We write $E \mo E'$ if $E$ is represented in the (well-founded) ultrapower of $V$ by $E'$. 
\end{definition}

The Mitchell order $\mo$ was introduced by Mitchell in \cite{Mit74} to construct canonical inner models with many measurable cardinals. The Mitchell order, which was initially introduced as an ordering on normal measures,  has been extended to extenders and plays a significant role in inner model theory.
As a prominent notion in the theory of large cardinals, the study of the Mitchell order and its structure has expanded in recent decades. 
The behaviour of the Mitchell order on extenders depends on the type of extenders in consideration and naturally becomes more complicated when restricted to stronger types of extenders. A fundamental dividing line in the behaviour of the Mitchell order is its well-foundedness:
Mitchell (\cite{Mit83}) has shown that $\mo$ is well-founded when restricted to normal measures. 
The question of the well-foundedness of $\mo$ was further studied by Steel \cite{St93}, and Neeman \cite{Ne04}, who showed that it fails exactly at the level of rank-to-rank extenders.

\begin{definition}\label{def:rtr}
 Let $E$ be an extender. We say $E$ is a
  \emph{rank-to-rank extender} iff assuming $\lambda > \crit(E)$ is least
  such that $i_E(\lambda) = \lambda$, then
  $V_\lambda \subseteq M_E$. 
\end{definition}
Due to their similarity with embeddings $j : V \to M$ with $V_{\lambda+2} \subseteq M$, which have been shown by Kunen to be inconsistent with ZFC, the large cardinal strength witnessed by rank-to-rank extenders is considered to be located near the top of the large cardinal hierarchy. More precisely, rank-to-rank extenders naturally arise from the large cardinal axiom I2.

The known dividing line of well-foundedness naturally breaks the question of the general possible behaviour into two (i) which well-founded partial orderings can be isomorphic to the Mitchell order on measures/extenders below the rank-to-rank level? and (ii) which ill-founded partial orderings can be isomorphic to the Mitchell order on a set of rank-to-rank extenders?\footnote{
Beyond the possible (partial) ordering structure of the Mitchell order, the investigation can be further extended to non-transitive relations, as the Mitchell order need not be transitive in general (see \cite{St93}). This direction is not developed in this paper.
}

Concerning question (i), the possible structure of the Mitchell order on normal measures has been extensively studied in
 \cite{Mit83}, \cite{Bal85}, \cite{Cum93}, \cite{Cum94},  \cite{Wit96}, \cite{BN16}, \cite{BN15}.  
 It has been shown in \cite{BN15} that it is consistent for every well-founded partial ordering to be isomorphic to the restriction of $\mo$ to the set of normal measures on some measurable cardinal $\kappa$ (the exact consistency strength of this property has not been discovered).
 
 In this work, we make a first step towards expanding the study of the Mitchell order in the ill-founded case, and address question (ii).  
 Specifically, we focus on the extent to which the well-foundedness of the Mitchell order fails on rank-to-rank extenders, by considering possible ordertypes of infinite decreasing chains in $\mo$. 
 The main results of this paper are the following two theorems. 

 \begin{theorem}\label{thm:ctbleseq} Assume there exists a rank-to-rank extender $E$. Then for every countable ordinal $\gamma$  there is a sequence of rank-to-rank extenders of length $\gamma$, $(E_\alpha \st \alpha < \gamma)$, on which the Mitchell order is transitive and strictly decreasing.
\end{theorem}

\begin{theorem}\label{thm:no.unctblseq}
 There is no $\omega_1$-sequence of extenders which is strictly decreasing and transitive in the Mitchell order.
\end{theorem}

Our presentation of the proof of Theorem \ref{thm:no.unctblseq} goes through a proof of a weak version of Steel's conjecture, which addresses \emph{transitive} $\omega$-sequences of extenders.\footnote{See Section \ref{sec:omega1seq} for a formulation of Steel's conjecture.} 
This presentation replaces a previous ad-hoc proof. The authors would like to thank Grigor Sargsyan for pointing out the connection with Steel's conjecture, which led to the current concise proof of Theorem \ref{thm:no.unctblseq}.  The (full) conjecture was recently proved by Goldberg (\cite{Go}) building on his remarkable study of the internal relation.\\ 
%We believe that our proof of the weaker statement is of interest due to the fact that it only employs elementary concepts of the theory of extenders. \\

\textbf{Acknowledgements} \\
This research was initiated whilst the second-listed author visited the first-listed author at the Hebrew University of Jerusalem in February 2019. She wishes to thank the Hebrew University of Jerusalem for its hospitality.
The first-listed author was partially supported by the Israel Science Foundation Grant 1832/19. The second-listed author gratefully acknowledges funding from L'OR\'{E}AL Austria, in collaboration with the Austrian UNESCO Commission and in cooperation with the Austrian Academy of Sciences - Fellowship \emph{Determinacy and Large Cardinals}. 
The authors would like to thank Grigor Sargsyan and the referee for many valuable comments and suggestions, which greatly improved the presentation of the paper. 

\section{Basic definitions and observations}\label{sec:basicdefinitions}

%We start with recalling the definition of \emph{extenders}. In general, extenders generalize measures and are able to code arbitrary elementary embeddings. A $(\kappa,\lambda)$-extender is a sequence
%  $E = (E_a \st a \in [\lambda]^{<\omega})$ of measures $E_a$ with $\crit(E) = \kappa$
 % and $\lh(E) = \lambda$. For an elementary embedding $j : V \rightarrow M$, the \emph{$(\kappa,\lambda)$-extender $E$ derived from $j$} is given by \[ X \in E_a \text{ iff } X \subseteq [\xi]^{|a|} \wedge a \in j(X), \] for each $a \in [\lambda]^{<\omega}$, where $\kappa$ is the critical point of $j$ and $\xi \geq \kappa$ is the least ordinal such that $\lambda \leq j(\xi)$. Then the $E_a$ give rise to a canonical system of ultrapowers $\Ult(V,E_a)$ with direct limit $\Ult(V,E)$. For a detailed introduction to extenders we refer the reader to \cite[Section 26]{Ka08}. We fix the following notation for $(\kappa,\lambda)$-extenders $E$. For $a \in [\lambda]^{<\omega}$, $\mu_a$
 % denotes the least $\mu$ such that $a \subseteq j(\mu)$, if we are
  %deriving an extender from $j$. \[i_E \colon V \rightarrow M_E\] denotes the ultrapower embedding by $E$. 
  
  For notational clarity it will be sometimes convenient for us to use the notation $\tUlt(V,E)$ for $M_E$, the transitive collapse of $\Ult(V,E)$.
 For an extender $E$ in a transitive class $N$, write $M_E^N$ for the transitive collapse of $\Ult(N,E)$ and $i_E^N : M \rightarrow M_E^N$ for the corresponding ultrapower embedding. For a rank-to-rank extender $E$ we write $\lambda^E$ for the least $\lambda > \crit(E)$ such that $i_E(\lambda) = \lambda$. Moreover, we write $\kappa^E_0 = \crit(E)$ and
  $\kappa^E_{n+1} = i_E(\kappa^E_n)$ for $n < \omega$, and call $(\kappa^E_n \st n < \omega)$ the \emph{critical sequence of $E$}. 
  For any rank-to-rank extender $E$, $\lambda^E = \sup_{n<\omega} \kappa^E_n = \lh(E)$. \\ 
  %(see \cite[Proposition 5.7]{Di18} for a proof). 
  For every $n < \omega$ let $E\res \kappa_n^E$ be the cutback of $E$ to the measures $E_a$, $a \in [\kappa_n^E]^{<\omega}$. 
  We have that $M_E = \bigcup_n X_n$ where \[X_n = \{ i_E(f)(a) \mid a \in [\kappa_n^E]^{<\omega}, f : [\mu_a]^{|a|} \to V \}.\] By taking $f = \id$, it is clear that $\kappa_n^E \subseteq X_n$. Let $N_n$ denote the transitive collapse of $X_n$ for each $n<\omega$. We have that $N_n$ is isomorphic to $M_{E \upharpoonright \kappa_n^E}$. Moreover, for each $n < \omega$, 
  $N_{n+1}$ contains $V_{\kappa_{n+1}^E}$ and in particular $i_E``\kappa_n^E \in N_{n+1}$. It follows from a standard argument that ${}^{\kappa_n^E}N_{n+1} \subseteq N_{n+1}$. As the critical point of the inverse of the collapse map $N_{n+1} \rightarrow X_{n+1}$ is $>\kappa_n^E$, this implies ${}^{\kappa_n^E}X_{n+1} \subseteq X_{n+1}$.
  
  The following observations will be useful in comparing extenders in different ultrapowers. 
  \begin{lemma}\label{lemma:subsets.of.Vlambda}
  Assume that $N = \bigcup_{n<\omega} X_n$ is an increasing union of classes $X_n$, $N$ is transitive, and there exists a sequence of cardinals $\l \kappa_n \mid n < \omega\r$,
  such that ${}^{\kappa_n} X_{n+1} \subseteq X_{n+1}$ for all $n < \omega$. 
  Let $E \in N$ be a rank-to-rank extender of height $\lh(E) = \lambda = \bigcup_n \kappa_n$. 
  \begin{enumerate}
   \item For every $a \in [\lambda]^{<\omega}$ and $f : [\mu_a]^{|a|} \to N$ a function in $V$,  there exists a function $f_N \in N$ such that 
  $\{ \nu \in [\mu_a]^{|a|} \mid f(\nu) = f_N(\nu)\} \in E_a$. 
  \item Let $i_E : V \to M_E$ and $i^N_E : N \to M_E^N$ be the ultrapower embeddings by $E$ of $V$ and $N$, respectively. For every set $x \in N$, $i_E(x) = i^N_E(x)$. 
  \end{enumerate}
\end{lemma}

  \begin{proof}
  \begin{enumerate}
   \item As $N = \bigcup_n X_n$ is an increasing union, and $E_a$ is $\sigma$-complete,  there exists some $n$ such that the set
  $A_n = \{ \nu \in [\mu_a]^{|a|} \mid f(\nu) \in X_n\}$ belongs to $E_a$ and $\mu_a < \kappa_n^E$. 
  Since $X_{n+1}$ is closed under $\kappa_n^E$-sequences, it follows that $f\uhr A_n : [\mu_a]^{|a|} \to X_n$ belongs to $X_{n+1} \in N$.
  The claim follows. 
  
\item Let $\pi : \Ult(V,E) \rightarrow M_E$ and $\pi^N : \Ult(N,E) \rightarrow M_E^N$ denote the transitive collapse embeddings. We have that in fact for any $[a,f]_E$ and $[a,f_N]_E$ for $a \in [\lambda]^{<\omega}$, $f : [\mu_a]^{|a|} \to x$ in $V$ and $f_N : [\mu_a]^{|a|} \to x$ in $N$ with $\{ \nu \in [\mu_a]^{|a|} \mid f(\nu) = f_N(\nu)\} \in E_a$, \[ \pi([a,f]_E) = \pi^M([a,f_N]_E). \] This follows from a straightforward induction on the rank of $[a,f]_E$ using that for any $[b,g]_E \in [a,f]_E$, $g$ takes values in $N$, as $N$ is transitive and we may assume that $g \in N$ by the first part of the lemma. Now this implies $i_E(x) = i^N_E(x)$. 
  \end{enumerate}
  \end{proof}

 \begin{lemma}\label{lem:MOintheultrapower}
   Suppose $E_2,E_1,E_0$ are three rank-to-rank extenders of the same length
   $\lambda = \lambda^{E_i}$, $i = 0,1,2$, such that $E_2$ is Mitchell
   order below $E_1$, and both $E_2,E_1$ are Mitchell order below
   $E_0$. Then $M_{E_0}$ also witnesses that $E_2$ is Mitchell order
   below $E_1$.
 \end{lemma}

\begin{proof}
  The fact $E_2 \mo E_1$ means that $E_2$ is represented in the
  $V$-ultrapower by $E_1$, by a function $f$ and a generator $a$ of
  $E_1$. Take $k < \omega$ such that
  $a \in [\kappa_{k+1}^{E_1}]^{<\omega}$. 

  The reason it is not immediate that the relation
  $E_2 \lhd E_1$ also holds in $M_{E_0}$ is that the
  function $f$ need not belong to $M_{E_0}$. However, we can argue
  that $M_{E_0}$ does see some witnessing function $f^*$ by using
  approximations. Assume without loss of generality that $f(\nu)$ is an extender for every $\nu \in \dom(f)$. Indeed, notice that for every $n < \omega$, the
  function $f_n$ with $\dom(f_n) = \dom(f)$ such that for every $\nu$,
  $f_n(\nu)$ is the restriction of the extender $f(\nu)$ to length
  $\kappa_{n}^{E_1}$. Clearly, $[a,f_n]^V_{E_1}$ represents the cut back
  of $E_2$ to length $\kappa_{n+1}^{E_1}$. Moreover, $f_n$ belongs to
  $M_{E_0}$ for every $n < \omega$ since $V_\lambda \subseteq M_{E_0}$.

  Let $\mathcal{E}(\eta)$ for some ordinal $\eta$ denote the set of all $(\kappa,\eta)$-extenders. Now, working in $M_{E_0}$ and utilizing the fact that both $E_1,E_2$
  belong to the model, we consider the tree $T$ of all pairs
  $(\tau,n)$ such that
  $\tau : [\kappa_k^{E_1}]^{|a|} \to \mathcal{E}(\kappa_n^{E_1})$ satisfies
  that $[a,\tau]^{M_{E_0}}_{E_1}$ represents the restriction of $E_2$ to length
  $\kappa_{n+1}^{E_1}$. The tree order $<_T$ is given by
  $(\tau,n) <_T (\tau',n')$ if $n < n'$ and $\tau'(\nu)$ extends
  $\tau(\nu)$ for all $\nu \in \dom(\tau)$.  It is clear that a
  cofinal branch in $T$ translates to a function $F$ for which
  $[a,F]_{E_1}$ represents $E_2$, and vice versa. Therefore $f \in V$
  witnesses that $T$ has a cofinal branch in $V$, and thus, by absoluteness of well-foundedness, it must also have one in $M_{E_0}$.
\end{proof}

Steel gives in \cite{St93} a folklore example that for rank-to-rank
extenders the Mitchell order need not be well-founded. We recall it here because some of the ideas will be used later.

\begin{proposition}[Folklore]\label{prop:1stexample}
  Let $E$ be a rank-to-rank extender. Then there is a strictly decreasing sequence of length $\omega$ in the Mitchell order on which $\mo$ is transitive.
\end{proposition}

\begin{proof}
  Consider the
  following sequence of rank-to-rank extenders
  $(E_n \colon n < \omega)$. Let $E_0 = E$ and
  $E_{n+1} = i_{E_n}(E_n)$, where
  $i_{E_n} \colon V \rightarrow M_{E_n}$ is the canonical
  embedding associated to $E_n$. Then it is straightforward to check that every $E_n$ is a $V$-extender
  and $E_{n+1} \lhd E_n$ for all $n<\omega$.

  \begin{claim}\label{cl:transitivity}
    The Mitchell order is transitive on $(E_n \colon n < \omega)$.
  \end{claim}
  \begin{proof}
    Let $n<\omega$. We show that $E_{n+2} \lhd E_{n}$, the rest follows
    analogously. By construction $E_{n+2} \in M_{E_{n+1}}$ and
    $E_{n+1} \in M_{E_n}$, we argue that
    $E_{n+2} \in M_{E_n}$. Argue that $i_{E_{n+1}}(E_{n+1}) = i_{E_{n+1}}^{M_{E_n}}(E_{n+1})$,
    where $i_{E_{n+1}}^{M_{E_n}} \colon M_{E_n} \rightarrow M_{E_{n+1}}^{M_{E_n}}$.
    By Lemma \ref{lemma:subsets.of.Vlambda}, applied to $E = E_{n+1}$ and $M_{E_n}$, we see that $E_{n+2} = i_{E_{n+1}}(E_{n+1}) = i_{E_{n+1}}^{M_{E_n}}(E_{n+1})$, and hence $E_{n+2} \in M_{E_n}$.
  \end{proof}
\end{proof}

\section{Countable decreasing sequences in the Mitchell order}\label{sec:ctbleseq}

In this section we prove Theorem \ref{thm:ctbleseq} and show that there can be strictly decreasing transitive sequences in the Mitchell order of any countable length.

Before we turn to the proof, we would like to emphasize why we cannot simply extend the construction in Proposition \ref{prop:1stexample} to obtain decreasing sequences of extenders in the Mitchell order of arbitrary (countable) length. Let $E$ be a rank-to-rank extender and let $(E_n \colon n < \omega)$ be the corresponding sequence of extenders constructed in the proof of Proposition \ref{prop:1stexample}. The critical points of the extenders $E_n$ are strictly increasing with supremum $\lambda^E$. A natural next extender to consider is the image of $E$ under the direct limit embedding obtained by successively applying the $E_n$‘s (internally, i.e., $E_1$ is applied to $M_{E_0}$ and so forth). This, however, does not work outright as this direct limit embedding moves $\kappa$, the critical point of $E$, to $\lambda$. Instead, we will construct a sequence of extenders $(E_n^\prime \colon n < \omega)$ such that the direct limit embedding $j$ obtained by successively applying the $E_n^\prime$‘s (again, internally) has critical point strictly above $\kappa$. Now $j(E)$ is a rank-to-rank extender with critical point $\kappa$ and $\lambda^{j(E)} = \lambda$. Moreover, $j(E) \mo E_n^\prime$ for all $n<\omega$, so if we let $E_\omega^\prime = j(E)$, $(E_n^\prime \colon n \leq \omega)$ is a decreasing sequence of length $\omega+1$ in the Mitchell order. The proof of Theorem \ref{thm:ctbleseq} elaborates on this idea.

\begin{proof}[Proof of Theorem \ref{thm:ctbleseq}]
  Let $E$ be a rank-to-rank extender with critical sequence
  $(\kappa_n \st n<\omega)$ and $\lambda^E = \lambda$. In what follows, all extenders will have the same length $\lambda$. We start by introducing some notation for sequences of extenders as constructed in the proof of Proposition \ref{prop:1stexample}. For a rank-to-rank extender $F$, write $S^1(F) = i_{F}(F)$ and $S^{n+1}(F) = i_{S^n(F)}(S^n(F))$. Now, the decreasing sequences in the Mitchell order we construct will be of the following form.

  \begin{definition}\label{def:guidedbyinternaliteration}
    Let $\vec E = (E_\alpha \st \alpha < \gamma)$ be a sequence of rank-to-rank
    extenders. Then we say that \emph{$\vec E$ is
      guided by an internal iteration} iff there are well-founded
    models $(M_\alpha \st \alpha \leq \gamma)$,
    $(M_\alpha^* \st \alpha \leq \gamma)$ with $M_0 = M_0^* = V$ and
    elementary embeddings
    $j^*_{\alpha,\beta} \colon M_\alpha^* \rightarrow M_\beta^*$ for all
    $\alpha < \beta \leq \gamma$ such that
    \begin{enumerate}
    \item $E_{\alpha+1} = S^n(E_\alpha)$ for some $n \geq 1$ and all
      $\alpha+1 < \gamma$, 
      \item $M_{\nu} = M_\nu^*$ for all limit ordinals
      $\nu \leq \gamma$ is given as the direct limit of the directed system 
      $\l M_\alpha^* ,j^*_{\alpha,\beta} \mid \alpha \leq\beta < \nu\r$. 
      
       \item $M_{\alpha+1} = M_{E_\alpha}$ and
      $M_{\alpha+1}^*$ is the transitive collapse of $\Ult(M_\alpha^*,E_\alpha)$. 
      For notational convenience (see the diagram below) we will use $M_{\alpha+1}^N$ to denote the transitive collapse of $\Ult(N, E_\alpha)$, therefore $M_{\alpha+1}^* = M_{\alpha+1}^{M_\alpha^*}$.
    \item for all limit ordinals
      $\nu < \gamma$, $E_\nu = j^*_{0,\nu}(E_0)$ and $\Ult(V, E_\nu)$ is
      well-founded, and 
    \item the following diagram commutes and all maps in the diagram are given by internal ultrapowers.
          \end{enumerate}
  \end{definition}
      \begin{figure}[htb]\label{Diagram}
      \begin{tikzpicture}
        \draw[->] (0,0) node[left] {$V$}-- (0.6,0) node[right]
        {$M_1^*$};
        \draw[->] (1.4,0) -- (2,0) node[right] {$M_2^*$};
        \draw[->] (2.8,0) -- (3.4,0) node[right] {$M_3^*$};
        \draw[->] (4.2,0) -- (4.8,0) node[right] {$\cdots \;
          M_\omega$};
        \draw[->] (6.2,0) -- (6.8,0) node[right] {$M_{\omega+1}^*$};
        \draw[->] (8,0) -- (8.6,0) node[right] {$M_{\omega+2}^*$};
        \draw[->] (9.8,0) -- (10.4,0) node[right] {$\cdots \; M_\gamma^*$};
        
        \draw[->] (1,-1) node[below] {$V$} -- (1,-0.25);
        \draw[->] (2.4,-1) node[below] {$M_2$} -- (2.4,-0.25);
        \draw[->] (3.8,-1) node[below] {\footnotesize $M_3^{M_2}$} -- (3.8,-0.25);
        \draw[->] (7.4,-1) node[below] {$M_\omega$} -- (7.4,-0.25);
        \draw[->] (9.1,-1) node[below] {\footnotesize $M_{\omega+2}^{M_\omega}$} --
        (9.1,-0.25);

        \draw[->] (1.4,-1.25) -- (2,-1.25);
        \draw[->] (2.8,-1.25) -- (3.4,-1.25);
        \draw[->] (4.2,-1.25) -- (4.8,-1.25) node[right] {$\cdots$};
        \draw[->] (8,-1.25) -- (8.6,-1.25);
        \draw[->] (9.8,-1.25) -- (10.4,-1.25) node[right] {$\cdots$};

        \draw[->] (2.4,-2.35) node[below] {$V$} -- (2.4,-1.6);
        \draw[->] (3.8,-2.35) node[below] {$M_3$} -- (3.8,-1.6);
        \draw[->] (9.1,-2.35) node[below] {$M_\omega$} --
        (9.1,-1.6);

        \draw[->] (2.8,-2.6) -- (3.4,-2.6);
        \draw[->] (4.2,-2.6) -- (4.8,-2.6) node[right] {$\cdots$};
        \draw[->] (9.8,-2.6) -- (10.4,-2.6) node[right] {$\cdots$};

        \draw[->] (3.8,-3.7) node[below] {$V$} -- (3.8,-2.95);
        \draw (4.8,-3.3) node {$\ddots$};
        \draw (10.4,-3.3) node {$\ddots$};
      \end{tikzpicture}
      \end{figure}

   The term \emph{guided by an internal iteration} refers to the fact that the iteration $\l M_\alpha^*, j^*_{\alpha,\beta} \mid \alpha \leq \beta < \gamma\r$ is internal, as shown by the next claim. 
    \begin{claim}
      For each $\alpha < \gamma$, $E_\alpha \in M_\alpha^*$. 
    \end{claim}
    \begin{proof}
      The claim is immediate when $\alpha$ is a limit ordinal or $\alpha = 0$. 
      Let $\alpha = \beta +1$ be a successor ordinal
      and assume inductively that $E_\beta \in M_\beta^*$. 
      
      Suppose first that $E_\alpha =S^n(E_\beta)$ for $n = 1$. 
      Then $E_\alpha  = i_{E_\beta}(E_\beta)$, where
      $i_{E_\beta} \colon V \rightarrow M_{E_\beta}$ is the
      $V$-ultrapower embedding given by $E_\beta$.
      Recall that
      $M^*_\alpha = M^{M_\beta^*}_{E_\beta}$. 
      By applying Lemma \ref{lemma:subsets.of.Vlambda} to 
      $M = M_\beta^*$ and $E = E_\beta$,  we conclude that 
      $E_\alpha = i_{E_\beta}^{M_\beta^*}(E_\beta) \in M^{M_\beta^*}_{E_\beta} =
      M_\alpha^*$.
      We now know that $S^1(E_\beta) \in M_{\alpha}^*$, and so for 
      $n >1$, we can further compute $(S^{n-1})^{M_\alpha^*}(S^1(E_\beta))$ inside $M_{\alpha}^*$, which by applying Lemma \ref{lemma:subsets.of.Vlambda} $(n-1)$ consecutive times, results in 
      \[
      (S^{n-1})^{M_\alpha^*}(S^1(E_\beta)) = S^n(E_\beta) = E_{\alpha}.
      \] \end{proof}

  Before we prove the existence of sequences of extenders guided by
  internal iterations, we show an abstract claim which will allow us
  to extend any such sequence by one further element. This is shown as
  in the proof of Theorem 2.2 in \cite{St93} (where this particular
  argument is attributed to Martin). We sketch the proof here for the
  reader's convenience.

    \begin{claim}\label{cl:internalultrapowers}
      Let $N \cong \Ult(V,E')$ for some rank-to-rank extender $E'$ of length $\lambda$, and $E \lhd F$ rank-to-rank extenders of length $\lambda$ with $E,F \in N$. Then the following diagram commutes and all maps are given by internal ultrapowers.
      \begin{figure}[htb]
      \begin{tikzpicture}
        \draw[->] (-0.8,0) node[left] {$M_F^N$}-- (0.6,0) node[right]
        {$\tUlt(M_F^N, E) = \tUlt(M_E^N, i_E(F))$}; 

        \draw[->] (-1,-1.25) node[left] {$N$}-- (3.7,-1.25) node[right]
        {$M_E^N$};
       
        \draw[->] (-1.3,-1) -- (-1.3,-0.25);
        \draw[->] (4.1,-1) -- (4.1,-0.25);
      \end{tikzpicture}
      \end{figure}
    \end{claim}
    \begin{proof}
      First, we note that $\tUlt(M_E^N, i_E(F)) = i_E(M_F^N)$, where we
      consider $M_F^N$ as a class of $N$. 
      Second, we assume $E \mo F$ in $V$, which by Lemma 
      \ref{lem:MOintheultrapower} implies that $E \mo F$ in $N$ as well, i.e.,  $E \in M_F^N$.
      We can therefore form the internal ultrapower of $M_F^N$ by $E$. 
      Let $i_E^{M_F^N} \colon M_F^N \rightarrow \tUlt(M_F^N,E)$ be the resulting ultrapower embedding.
      Finally, we have that every $x \in i_E(M_F^N)$ can be identified with equivalence  classes $[a,f]_E$, where $a \in [\lambda]^{<\omega}$ and
      $f \colon [\mu_a]^{|a|} \rightarrow M_F^N$, and by Lemma \ref{lemma:subsets.of.Vlambda},  $[a,f]_{E} = [a,g]_E$ for some function $g \in M_F^N$. 
       It follows that
      $i_E(M_F^N)$ identifies with the internal ultrapower
      $\Ult(M_F^N,E)$, and 
      $i_E \upharpoonright M_F^N$ with  $i_E^{M_F^N}$.
    \end{proof}

    Next, we argue that sequences as in Definition \ref{def:guidedbyinternaliteration} are in fact as desired.

  \begin{claim}
    If a sequence $\vec E = (E_\alpha \st \alpha < \gamma)$ of
    extenders is guided by an internal iteration, the embeddings in the
    diagram witness that $E_\beta \lhd E_\alpha$ for all
    $\alpha < \beta < \gamma$. In particular,  $\vec E$ is a decreasing transitive sequence in $\mo$.
  \end{claim} 
  
  \begin{proof}

	Given a finite ordinal $n < \omega$ and a countable ordinal $\alpha > n$, 
	the last Subclaim implies that $E_\alpha \lhd E_n$. This is since 
	 all maps in Diagram \ref{Diagram} witnessing that $\vec E$ is guided by an internal iteration, are internal ultrapowers, and therefore there are internal iterations from $M_{n+1}$ (the ultrapower of $V$ by $E_n$) to $M_{n+1}^*$ (i.e., vertical maps), and from $M_{n+1}^*$ to $M_\alpha^*$, from which $E_\alpha$ is taken.
	 
    More generally, given countable ordinals $\beta < \alpha$, we see 
    that $E_\alpha \lhd E_\beta$ by moving first to the last limit ordinal $\eta \leq \beta$ (i.e.,  $\beta = \eta +n$ for some $n < \omega$), and apply the proof of Claim
    \ref{cl:internalultrapowers} to show that the following diagram
    commutes and its top row is given by an internal iteration of
    $M_{E_{\eta+n}} = \tUlt(V, E_{\eta+n})$.

    \begin{figure}[htb]
      \begin{tikzpicture}
        \draw[->] (0,0) node[left] {$\tUlt(V, E_{\eta+n})$}-- (0.6,0) node[right]
        {$\tUlt(M^*_1, E_{\eta+n})$};
        \draw[->] (3.3,0) -- (3.9,0) node[right] {$\tUlt(M_2^*, E_{\eta+n})$};
        \draw[->] (6.6,0) -- (7.2,0) node[right] {$\cdots \;
          \tUlt(M^*_\eta, E_{\eta+n})$};

        \draw[->] (-1,-1.25) node[left] {$V$}-- (1.6,-1.25) node[right]
        {$M^*_1$};
        \draw[->] (2.4,-1.25) -- (4.9,-1.25) node[right] {$M_2^*$};
        \draw[->] (5.7,-1.25) -- (8.2,-1.25) node[right] {$\cdots \; M^*_\eta$};

        \draw[->] (-1.3,-1) -- (-1.3,-0.25);
        \draw[->] (2,-1) -- (2,-0.25);
        \draw[->] (5.3,-1) -- (5.3,-0.25);
        \draw[->] (9.2,-1) -- (9.2,-0.25);
      \end{tikzpicture}
    \end{figure}

    \iffalse
    \begin{figure}[htb]
      \begin{tikzpicture}
        \draw[->] (-1,0) node[left] {$M^*_\eta$}-- (0.6,0) node[right]
        {$\Ult(M^*_\eta, E_{\eta+n})$};
        
        \draw[->] (-1.1,-1.25) node[left] {$V$}-- (0.8,-1.25) node[right]
        {$\Ult(V, E_{\eta+n})$};

        \draw[->] (-1.4,-1) -- (-1.4,-0.25);
        \draw[->] (1.8,-1) -- (1.8,-0.25);
      \end{tikzpicture}
    \end{figure}
    \fi
    
    The extenders used to obtain the bottom row of the diagram are $E_0$, $E_1$, $E_2$, $\dots$, while the extenders used to obtain the top row of the diagram are 
    $i_{E_{\eta+n}}(E_0), i^{M^*_1}_{E_{\eta+n}}(E_1),i^{M_2^*}_{E_{\eta+n}}(E_2), \dots$.\\ 

    Finally, we get from $M_\eta^*$ to $M_{\eta+n}^*$ using the same argument described above, for $n < \alpha$. This suffices to show that
    $E_\alpha \lhd E_{\eta+n} = E_\beta$.
    \end{proof}

  We now turn to proving the following claim, which immediately
  yields the theorem. 

  \begin{claim}
    Assume there is a rank-to-rank extendere $E$. 
    For every countable ordinal $\gamma < \omega_1$ and an ordinal $\kappa < \lambda = \lambda^E$, there
    is a sequence of rank-to-rank extenders
    $\vec E = (E_\alpha \st \alpha \leq \gamma)$ 
    which is guided by an internal iteration, such that the induced embedding 
    $j^*_{0,\gamma} \colon V \rightarrow M_\gamma^*$ satisfies
    $\kappa < \crit(j^*_{0,\gamma}) < \lambda$ and
    $j^*_{0,\gamma}(\lambda) = \lambda$.
  \end{claim}

  \begin{proof}
    We prove the claim by induction on $\gamma < \omega_1$. There is nothing to show for $\gamma = 0$. 
    Suppose that $\gamma = 1$. Let $\kappa < \lambda$
    be arbitrary and fix some $n<\omega$ such that $\kappa < \kappa_n^E$. 
    Then $E_0 = S^{n+1}(E)$
    and $E_1 = S^1(E_0)$ giving rise to
    $j^*_{0,1} \colon V \rightarrow M_{E_{0}}$ are as desired.

    Now, suppose $\gamma = \alpha + 1$ and the claim holds for
    $\alpha$ witnessed by $(E_\nu \st \nu \leq \alpha)$. Let $n$ be
    such that $\kappa < \kappa_n^{E_\alpha}$, the $n$-th element of
    $E_\alpha$'s critical sequence, and
    $E_{\alpha+1} = S^n(E_\alpha)$. Then
    $\crit(E_{\alpha+1}) \geq \kappa_n^\alpha$ and hence, using that
    inductively $\crit(j^*_{0, \alpha}) > \kappa$, we have
    $\crit(j^*_{0, \alpha+1}) > \kappa$. Using Claim
    \ref{cl:internalultrapowers} the extended sequence
    $(E_\nu \st \nu \leq \alpha+1)$ is as desired.

    Finally, suppose $\gamma < \omega_1$ is a limit ordinal and  fix an increasing sequence $\l\alpha_n \mid n < \omega\r$, cofinal in $\gamma$, with $\alpha_0 = 0$. 
    We also fix a well-ordering $<_w$ of $H_{\lambda^+}$ in $V$. 
     
     By the inductive hypothesis
    applied to $\alpha_1$, there is a sequence $\vec E^0$ of rank-to-rank extenders which is guided by an internal iteration and
    an elementary embedding
    $j^*_{0,\alpha_1} \colon V \rightarrow M_{\alpha_1}^*$ with critical
    point $\nu_0 > \kappa$ and $j^*_{0,\alpha_1}(\lambda)=\lambda$. 
    We pick $\vec E^0$ to be the minimal such sequence with respect to $<_w$. 
    By
    elementarity of $j^*_{0,\alpha_1}$, we can apply the inductive
    hypothesis again inside $M_{\alpha_1}^*$ to get a sequence
    $\vec E^1$ of rank-to-rank extenders which is guided by an
    internal iteration and an elementary embedding
    $j^*_{\alpha_1,\alpha_2} \colon M_{\alpha_1}^* \rightarrow
    M_{\alpha_2}^*$ with critical point
    $\nu_1 > j^*_{0,\alpha_1}(\nu_0)$ and
    $j^*_{\alpha_1,\alpha_2}(\lambda)=\lambda$. 
    We take $\vec E^1$ to be the minimal such sequence in $M_{\alpha_1}^*$, with respect to 
    $j_{0,\alpha_1}^*(<_w)$. 
    Repeating this procedure yields a
    sequence $((\vec E^n, j^*_{\alpha_n,\alpha_{n+1}}) \st n < \omega)$
    of sequences of extenders together with elementary embeddings. 
    
    Let
    $(E_\alpha \st \alpha < \gamma)$ be the concatenation of the sequences 
    $\vec E^n$, $n < \omega$. 
    The choice of $\vec E^{n+1}$ to be minimal in $M_{\alpha_{n}}^*$ with respect to the well ordering $j_{0,\alpha_n}^*(<_w)$ guarantees that the 
    sequence $\l \vec E^m \mid m > n\r$ belongs to 
    $M_{\alpha_n}^*$, and thus also the tail of the iteration $\l M_{\alpha}^*, j^*_{\alpha,\beta} \mid \alpha_n< \alpha \leq \beta < \gamma\r$.  
    Note that $\crit(j^*_{0,\gamma}) > \kappa$ since
    $\crit(j^*_{\alpha_n, \alpha_{n+1}}) > \kappa$ for all $n < \omega$. 
      Let $j^*_{0,\gamma} \colon V \rightarrow M_\gamma^* = M_\gamma$ be the
    direct limit embedding of the system.
    
    The \emph{reflecting a minimal counterexample} argument used to show
    that internal iterations by normal ultrafilters are well-founded
    (see e.g., \cite[Theorem 19.7]{Je03} for normal ultrafilters or \cite[Proposition
    5.8]{Di18} for rank-to-rank extenders), can also be used to show that $M_\gamma$ is
    well-founded.

    \begin{subclaim}
      $j^*_{0,\gamma}(\lambda) = \lambda$.
    \end{subclaim}
    \begin{proof}
      Suppose not. Then there is some $\eta < \lambda$ such that
      $j(\eta) \geq \lambda$. But for every $\eta < \lambda$, there is
      by choice of the embeddings $j^*_{\alpha_n,\alpha_{n+1}}$ some
      $n < \omega$ such that $\nu_n = \crit(j^*_{\alpha_n,\alpha_{n+1}}) > \eta$,
      i.e. $j^*_{\alpha_n,\gamma}(\eta) = \eta$.
    \end{proof}

    Moreover, it is clear from the construction that, letting
    $E_\gamma = j^*_{0,\gamma}(E)$, the resulting sequence
    $(E_\alpha \st \alpha \leq \gamma)$ of rank-to-rank extenders is guided by an internal iteration. The only condition that
    needs a small argument is the following.

    \begin{subclaim}
      $\Ult(V,E_\gamma)$ is well-founded.
    \end{subclaim}
    \begin{proof}
      As $M_\gamma$ is well-founded and $E_\gamma$ is an extender in
      $M_\gamma$, $\Ult(M_\gamma, E_\gamma)$ is well-founded. We prove
      the subclaim by defining an elementary embedding $\pi \colon
      \Ult(V,E_\gamma) \rightarrow \Ult(M_\gamma, E_\gamma)$ as
      follows. For $[a,f]_{E_\gamma}^V \in \Ult(V,E_\gamma)$, let
      \[ \pi([a,f]_{E_\gamma}^V) = [a, j^*_{0,\gamma}(f) \circ
        j^*_{0,\gamma}]_{E_\gamma}^{M_\gamma}. \] 
        $\pi$ is
      well-defined since for $a \in [\lambda]^{<\omega}$,
      $\mu_a < \lambda$. Therefore,
      $j^*_{0,\gamma} \upharpoonright [\mu_a]^{|a|} \in V_\lambda
      \subseteq M_\gamma$. In addition,
      $j^*_{0,\gamma}(f) \circ j^*_{0,\gamma} = j^*_{0,\gamma} \circ f$, so
      $\pi$ is elementary.
    \end{proof}
  \end{proof}
  This finishes the proof of Theorem \ref{thm:ctbleseq}.
\end{proof}

\section{A bound on the length of decreasing sequences in the Mitchell
  order}\label{sec:omega1seq}

Steel proved in \cite{St93} that in a Mitchell order decreasing sequence of rank-to-rank extenders, the extenders cannot all have the same critical point. He conjectured the following stronger statement.

\begin{conjecture}[Steel]
  Suppose that $(E_m \st m < \omega)$ is a sequence of rank-to-rank extenders which is strictly decreasing in $\lhd$. Let $\lambda$ be the unique ordinal such that $\lambda = \lambda^{E_m}$ for all sufficiently large $m$. Then $\sup_{m < \omega} \crit(E_m) = \lambda$.
\end{conjecture}

Theorem \ref{thm2} below establishes Steel's Conjecture for the special case that the Mitchell order is \emph{transitive} on the sequence $(E_m \st m < \omega)$.

\begin{theorem}\label{thm2}
 Suppose that $(E_m \st m < \omega)$ is a sequence of rank-to-rank extenders, which is strictly decreasing and transitive in $\lhd$. Let $\lambda$ be the unique ordinal such that $\lambda = \lambda^{E_m}$ for all sufficiently large $m$. Then $\sup_{m < \omega} \crit(E_m) = \lambda$.
\end{theorem}

\begin{proof}
Suppose otherwise, and let  $\gamma_0$ be the minimal ordinal  for which there exists a $\lhd$-decreasing and transitive sequence $\vec{E} = (E_m \st m < \omega)$ such that $\gamma_0 = \sup_{m<\omega} \kappa_0^{E_m}  < \lambda^{\vec{E}} = \sup_{m<\omega} \lambda^{E_m}$. We assume without loss of generality that $\lambda^{\vec{E}} = \lambda^{E_m}$ for all $m < \omega$.
Let $n < \omega$ be the integer for which $\kappa_n^{E_0} \leq \gamma_0 < \kappa_{n+1}^{E_0}$. 
We move to the ultrapower $M_{E_0}$. By our assumption, $E_m \in
M_{E_0}$ for every $m > 0$, and by Lemma \ref{lem:MOintheultrapower}, $M_{E_0}$ sees that $E_m$ is Mitchell order below $E_k$ for every $0 < k < m < \omega$.
Since $M_{E_0}$ is not closed under $\omega$-sequences of its elements (in $V$) the sequence $(E_m \mid 1 \leq m < \omega)$ need not belong to $M_{E_0}$. Nevertheless, we may define in $M_{E_0}$ the tree $T^*$ of all finite sequences of rank-to-rank extenders $(E_m^* \mid m < N)$, which are strictly Mitchell order decreasing, transitive, have length $\lambda^{\vec{E}}$, and satisfy 
$\gamma_0 \geq \max_{m<N} \kappa_0^{E_m^*}$. 
The sequence $(E_m \mid 1 \leq m < \omega)$ witnesses that $T^*$ has a cofinal branch in $V$. So by absoluteness of well-foundedness there is a cofinal branch  in $M_{E_0 }$ as well.
We can now reflect this from $M_{E_0}$ back to $V$. Using the fact that $\kappa_n^{E_0} \leq \gamma_0 < \kappa_{n+1}^{E_0} = i_{E_0}(\kappa_n^{E_0})$, we conclude  that in $V$, there exists some $\gamma_{-1} < \kappa_n^{E_0} \leq \gamma_0$, and a sequence $\vec{E}^*= (E_m^* \mid m < \omega)$ of rank-to-rank extenders which is strictly $\lhd$-decreasing and transitive such that
$\sup_{m<\omega} \kappa_0^{E_m^*} \leq \gamma_{-1} < \lambda^{\vec{E}^*}$. This is a contradiction to the minimality of $\gamma_0$.
\end{proof}

Now we can obtain Theorem \ref{thm:no.unctblseq} as a corollary.

\begin{proof}[Proof of Theorem \ref{thm:no.unctblseq}]
 Suppose otherwise. Let $\vec{E} = (E_\alpha \st \alpha < \omega_1)$ be a sequence of extenders which is strictly decreasing and transitive in the Mitchell order. We may assume that all $E_\alpha$ are rank-to-rank extenders and that there exists some $\lambda^{\vec{E}}$ such that $\lambda^{E_\alpha} = \lambda^{\vec{E}}$ for all $\alpha < \omega_1$. In particular, $\cf(\lambda^{\vec{E}}) = \omega$ and we may choose a cofinal sequence $(\rho_n \st n < \omega)$ in $\lambda^{\vec{E}}$. By a straightforward pressing down argument, we can find an uncountable set $I \subseteq \omega_1$ and some $n^* < \omega$ such that
 $\kappa_0^{E_\alpha} < \rho_{n^*}$ for all $\alpha \in I$. 
 Taking $(\alpha_n \st n < \omega)$ to be the first $\omega$ many ordinals of $I$, it follows that $(E_{\alpha_n} \st n < \omega)$ is strictly decreasing and transitive in the Mitchell order, with $\sup_{n<\omega} \kappa_0^{E_{\alpha_n}} \leq \rho_{n^*} < \lambda^{\vec{E}}$. This contradicts Theorem \ref{thm2}. 
\end{proof}

\section{Questions}

After studying the length of the Mitchell order for rank-to-rank extenders, a natural question that arises is about the structure this order can have. 

\begin{question}
  Suppose there is a rank-to-rank extender. Can the tree order on the infinite binary tree $2^{<\omega}$ be realized by a Mitchell order? 
\end{question}

We can even ask the following more general question.

\begin{question}
  Suppose there is a rank-to-rank extender. Can any tree order on $\omega$ be realized by a Mitchell order?
\end{question}

\bibliographystyle{plain}
\bibliography{References}

\end{document}